\documentclass[12pt]{amsart}
\usepackage{amsfonts, amssymb, amscd, latexsym, graphicx, psfrag, color, float, enumitem}
\usepackage[all]{xy}
\usepackage{mathrsfs}
\usepackage[mathscr]{euscript}
\usepackage{mathabx}
\usepackage{tikz}

\usepackage[hyphens]{url}

\usepackage{todonotes}
\usepackage{hyperref}

\usepackage[hyphenbreaks]{breakurl}

\usepackage{tkz-euclide}
\usepackage{caption}

\newtheorem{dummy}{dummy}[section]
\newtheorem{lemma}[dummy]{Lemma}
\newtheorem{theorem}[dummy]{Theorem}
\newtheorem{innercustomthm}{Theorem}
\newenvironment{customthm}[1]
{\renewcommand\theinnercustomthm{#1}\innercustomthm}
  {\endinnercustomthm}

\newtheorem{corollary}[dummy]{Corollary}
\newtheorem{proposition}[dummy]{Proposition}
\newtheorem*{theorem*}{Theorem}
\newtheorem{definition}[dummy]{Definition}
\newtheorem*{definition*}{Definition}
\newtheorem{example}[dummy]{Example}

\newtheorem{remark}[dummy]{Remark}

\newcommand{\mHP}{\mathrm{HP}}
\newcommand{\mHHl}{\mathrm{HH}^{\mathrm{lin}}}

\newcommand{\bfA}{\mathsf{A}}
\newcommand{\bA}{\mathbb{A}}
\newcommand{\bC}{\mathbb{C}}
\newcommand{\bG}{\mathbb{G}}
\newcommand{\bL}{\mathbb{L}}

\newcommand{\bZ}{\mathbb{Z}}

\newcommand{\cC}{\mathcal{C}}

\newcommand{\cF}{\mathcal{F}}
\newcommand{\cO}{\mathcal{O}}
\newcommand{\cL}{\mathcal{L}}

\newcommand{\cK}{\mathcal{K}}
\newcommand{\cCK}{\mathcal{CK}}
\newcommand{\cX}{\mathcal{X}}

\newcommand{\cHH}{\mathcal{HH}}
\newcommand{\cHHl}{\mathcal{HH}^{\mathrm{lin}}}
\newcommand{\cH}{\mathcal{H}}

\newcommand{\cHP}{\mathcal{HP}}
\newcommand{\cHPl}{\mathcal{HP}^{\mathrm{lin}}}

\newcommand{\ft}{\mathfrak{t}}
\newcommand{\fg}{\mathfrak{g}}

\newcommand{\Map}[2]{\underline{\mathrm{Map}}\left (#1,#2\right )}

\newcommand{\dSt}{\mathrm{dSt}}

\newcommand{\Aff}[1]{\mathrm{Aff}(#1)}
\newcommand{\Spec}{\mathrm{S}\mathrm{pec}\,}

\newcommand{\DR}{\mathrm{DR}}

\newcommand{\El}{\mathcal{E}ll}
\newcommand{\an}{\mathrm{an}}

\newcommand{\Mod}{\mathrm{Mod}}
\newcommand{\Tot}{\mathrm{Tot}}

\newcommand{\Perf}{\mathrm{Perf}}
\newcommand{\Frac}{\mathrm{Frac}\,}

\newcommand{\Qcoh}{\mathrm{Qcoh}}
\newcommand{\QCoh}{\mathrm{QCoh}}
\newcommand{\Sym}{\mathrm{Sym}}

\newcommand{\Coh}{\mathrm{Coh}}

\begin{document}

\author[Tomasini]{Paolo Tomasini}
\address{Paolo Tomasini, SISSA\\
Via Bonomea 265\\ 34136 Trieste TS\\
Italy}
\email{\href{mailto:ptomasin@sissa.it}{ptomasin@sissa.it}}

\title[Adelic Descent for Equivariant Elliptic Cohomology]{Adelic Descent for Equivariant Elliptic Cohomology}

\begin{abstract}
We define $k$-rationalized $G$-equivariant elliptic cohomology, for a field of characteristic zero $k$ and a compact Lie group $G$, via adelic descent. We also give adelic descriptions of rationalized $G$-equivariant singular cohomology and K-theory. This completes a program first proposed by Ro\c{s}u. These descriptions are then used to obtain comparison results with periodic cyclic cohomologies defined via derived algebraic geometry.  
\end{abstract}

\maketitle

\tableofcontents

\section{Introduction}
A   powerful method to study complexified  equivariant cohomology theories is to regard them as \emph{quasi-coherent} sheaves of algebras on the decompletion of the formal group associated to their non-equivariant incarnation. Based on this principle,  Grojnowski  proposed a first construction of complexified equivariant elliptic cohomology in   \cite{groj}. At around the same time Ginzburg--Kapranov--Vasserot  gave an axiomatic description of equivariant elliptic cohomology in \cite{GKV}.  
The details of Grojnowski's construction were worked out by Ro\c{s}u in \cite{Ros1}, who applies similar methods in \cite{Ros3} to build complexified equivariant K-theory. The upshot is that elliptic cohomology and K-theory can be constructed  complex-analytically from the equivariant cohomology of fixed point loci. In this same paper, Ro\c{s}u states that such a description of complexified equivariant K-theory can also be obtained algebraically, via completions. Clearly, a byproduct of this construction would be a lift from the complexification to rationalization with respect to a general field of characteristic zero. Unfortunately, Ro\c{s}u never completed this program. 

This is what we are set to do in this paper: our goal is to obtain a purely algebraic description of equivariant K-theory and of Grojnowski's equivariant elliptic cohomology via adelic methods. In the last forty years elliptic cohomology has been intensely studied for its importance in homotopy theory --- for example Greenlees' approach via his algebraic model for rational $G$-spectra \cite{GreenRat} and Ganter's \cite{Gan14} --- and its relevance in mathematical physics and the theory of representations of loop groups (see Witten's \cite{WitQuant,WitDir} for the relationship with string theory and Dirac operators on loop spaces, or Ando's \cite{Ando} that highlights the relationship with loop groups). Nevertheless, the original question by Ro\c{s}u of a purely algebraic construction of its rationalized equivariant version remains unanswered. 

We will achieve this description via \emph{adelic descent}. Adelic descent has its origins in algebraic number theory and algebraic geometry, as a tool to study curves. One of the earliest applications  of adelic methods is a celebrated theorem of Andr\'e Weil, that describes principal $G$-bundles on curves in terms of the adelic ring of the curve. This theory was generalized to $n$-dimensional Noetherian schemes by Parshin \cite{ParAd} and Beilinson \cite{BeilAd}. Broadly speaking, adelic descent yields a description of coherent sheaves on Noetherian schemes in terms of their completion at chains of points on the scheme. A recent formulation    is due to Groechenig \cite{Groech}, who proved an \emph{adelic reconstruction theorem} in terms of an equivalence of $\infty$-categories of perfect complexes on a Noetherian scheme $X$ and perfect complexes on the adelic decomposition of $X$.

\subsection*{Rationalized equivariant elliptic cohomology}
In Section \ref{section:adelicGroj} we apply  adelic descent to study $G$-equivariant elliptic cohomology for compact Lie groups $G$.  Our approach is encoded in Definition \ref{def:elln} and Definition \ref{def:ellnG}. The key case is when $G=T$ is a torus. Here the construction proceeds by induction on the rank of $T$.  
If $k=\bC$ is the field of the complex numbers, our definition recovers Grojnowski's: this is the content of Theorem \ref{thm:equivGroj}.
Finally, we provide a chain-level presentation of Grojnowski's sheaf: this is carried out in the remarks in Subsection \ref{ssec:chains}. This presentation uses the same inductive construction of Definition \ref{def:elln} and rests on the formality of the algebra $C^\ast(BT)$ of cochains on the classifying space of $T$.

\subsection*{Adelic descent for equivariant cohomology and K-theory}
In Section \ref{section:adelicsingK} we complete Ro\c{s}u's program by giving a description of rationalized equivariant K-theory in terms of  adelic descent data, which is incarnated in Theorem \ref{thm:equivRosu}. This shows that rationalized equivariant K-theory can be constructed in a purely algebraic manner out of the equivariant cohomology of fixed loci, thus  implementing Ro\c{s}u's proposal. Our construction has several advantages, with respect to Ro\c{s}u's original paper \cite{Ros3}: it works over all fields of characteristic zero, and it recovers equivariant K-theory directly, rather than its extension by the holomorphic functions over the complexification of the torus.

We perform similar computations also for equivariant cohomology.
In particular this shows that, up to rationalization, equivariant elliptic cohomology, equivariant K-theory and equivariant cohomology can ultimately be built out of some basic local data --- always expressed in terms of Borel-equivariant singular cohomology of fixed loci --- and an induction based on localization theorems. This implements the philosophical picture determined by the Atiyah--Segal completion theorem.

\subsection*{Comparisons with periodic cyclic counterparts}
In Section \ref{section:geometry} we compare rationalized equivariant cohomology and K-theory with \emph{periodic cyclic homology theories} constructed respectively from the \emph{shifted tangent bundle} and the \emph{derived loop space} of a quotient stack.
The case of rationalized equivariant K-theory was considered by Halpern-Leistner--Pomerleano in \cite{HLPom}. There, they establish an equivalence between the periodic cyclic homology of quotient stacks $[X/G]$ of smooth quasi-projective schemes over $\bC$ admitting a semi-complete KN stratification and the $G$-equivariant topological K-theory of the analytification of $X$. 
Here we use adelic descent to prove the following theorem:
\begin{customthm}{A}
[Theorem \ref{thm:HPKG}]
\label{ctmthm:introKth}
Let $X$ be a smooth quasi-projective variety over $\bC$ acted on by a reductive group $G$. There is an equivalence of $\bZ_2$-periodic coherent sheaves on the GIT adjiont quotient $G//G$
$$\pi_\ast\cHP([X/G])\simeq\cK_{G^{\an}}(X^{\an})$$
which is natural in $X$ with respect to $G$-equivariant maps. Here, the left-hand side is periodic cyclic homology and the right-hand side is $G^\an$-equivariant K-theory of the analytification of $X$.
\end{customthm}

Our approach differs from that of \cite{HLPom}, as our proof makes use only of more elementary mathematics --- mainly adelic descent and localization theorems. Localization theorems in equivariant K-theory are quite classical and date back to the work of Segal \cite{SegK}, while localization in periodic cyclic homology is due to Chen \cite{HChen}.
Our techniques also allow us to drop the hypothesis of having a semi-complete KN stratification on the quotient $[X/G]$.

Halpern-Leistner and Pomerleano remark that their theorem follows from an identification at the level of cochains. Theorem \ref{ctmthm:introKth} has a similar lift, which requires a cochain model for equivariant K-theory. Such a model can be constructed again using adelic descent with the same arguments made in the elliptic setting in Section \ref{section:adelicGroj}. Such chain-level lifts appear in Proposition \ref{thm:HPKchains} (for algebraic tori) and Remark \ref{remark:HPKGchains} (for reductive $G$).

As a corollary of Theorem \ref{ctmthm:introKth} and Theorem 2.10 of \cite{HLPom}, we obtain a version of the lattice conjecture for categories of perfect complexes over quotient stacks of varieties by actions of reductive groups. Corollary \ref{cor:lattice} provides such version of the lattice conjecture, i.e. the existence of a rational structure on the periodic cyclic homology of a DG-category over the complex numbers. Also in our situation the rational structure comes from Blanc's topological K-theory \cite{Blanc}.

We remark that similar comparison results are known in differential geometry since the 90's. In this context, the equivariant K-theory of a smooth manifold can be recovered from the periodic cyclic homology of the algebra of $C^\infty$-functions on the manifold itself. Relevant references include \cite{BrylCycHom} and \cite{BlockGetz}.

A similar picture holds in the case of equivariant cohomology, where the \emph{shifted tangent stack} takes on the role played by the loop space in the case of K-theory. This story is well-known, and follows immediately from work of Pantev--To\"en--Vaqui\'e--Vezzosi \cite{PTVV} and Calaque--Pantev--To\"en--Vaqui\'e--Vezzosi \cite{CPTVV}. We include it in the paper, as we give a different argument based on  adelic methods. This allows to treat on the same footing equivariant elliptic cohomology, equivariant K-theory and equivariant cohomology. 

\begin{customthm}{B}[Theorem \ref{thm:HPlHG}]
Let $X$ be a smooth quasi-projective variety acted on by a reductive group $G$. There is an isomorphism of $\bZ_2$-periodic perfect complexes on the GIT adjoint quotient of the Lie algebra $\fg//G$
$$\pi_\ast\cHPl([X/G])\simeq\cH_{G^{\an}}(X^{\an})$$
where the right-hand side is $G^\an$-equivariant cohomology of the analytification of $X$, and the left-hand side is \emph{linearized} periodic cyclic homology (i.e. the Tate fixed points of the mixed structure on the de Rham complex of $[X/G]$).
\end{customthm}

A necessary requirement for the proof via adelic descent is a \emph{localization formula} for the shifted tangent stack, Proposition \ref{prop:shiftedloc}. To the best of the author's knowledge, this localization phenomenon does not appear in the literature and might be of independent interest. Localization for the loop space has been extensively studied in \cite{HChen}.

The case of equivariant elliptic cohomology has been investigated in \cite{MapStI}, where the geometric object taking the role of the loop space and the shifted tangent stack is an appropriate derived stack of \emph{quasi-constant} maps from an elliptic curve $E$ over a field of characteristic zero. The appropriate notion of \emph{elliptic periodic cyclic homology} is introduced, and proved to be equivalent to Grojnowski's equivariant elliptic cohomology when $k=\bC$. Thus our results in this paper, together with \cite{MapStI}, provide a unified treatment of (rationalized) equivariant elliptic cohomology, equivariant K-theory and equivariant cohomology and their comparisons with corresponding invariants defined via derived algebraic geometry. 

{\bf Acknowledgements:} 
I would like to thank Nicol\`o Sibilla for numerous discussions we had on the topic of this paper, and for reading a preliminary draft. I deeply thank Pavel Safronov for pointing out some issues with a previous version of Section \ref{ssection:HPlin} and Lemma \ref{lemma:trivaction} and for clarifying the role of $B\widehat{\bG}_a$.
I also thank Joost Nuiten, Emanuele Pavia and Bertrand To\"en for important discussions on topics related to the content of this paper.

\section{Preliminaries}
\label{section:preliminaries}

For the preliminary section of this paper we refer to the Preliminaries section in the paper \cite{MapStI}. For the reader's convenience, we review some  basic material.

\subsection{Complexified Equivariant Elliptic Cohomology}\label{subsection:PerlimGroj} 
Complexified equivariant elliptic cohomology was axiomatically defined by Ginzburg--Kapranov--Vasserot in \cite{GKV} and constructed by Grojnowski in \cite{groj}.
We follow mostly the more recent exposition found in \cite{Gan14} and \cite{SchSib}. Other reviews closer in style to the original can be found in \cite{And1}, \cite{Ros1} and \cite{GreenRat}. We remark that Grojnowski's paper only sketches the construction, but the details were carried out by Ro\c{s}u in \cite{Ros1}. 

Let $X$ be a finite $T$-CW-complex, where $T$ is real torus of rank $n$. Complex $T$-equivariant elliptic cohomology of $X$ is defined by first constructing a coherent sheaf of $\bZ_2$-graded algebras $\El^{\an}_T(X)$ over the complex manifold 
$$E_T := E\otimes_{\bZ}\check{T}$$
and then viewing it as an algebraic coherent sheaf via standard GAGA arguments, yielding
$$
\El_T(X) \in \Qcoh(E_T)^{\mathbb{Z}_2}
$$ 
The construction we review is as a $\bZ_2$-periodic rather than $\bZ_2$-graded coherent sheaf. The two constructions are completely equivalent.

Grojnowski's main insight is that, rationally, equivariant elliptic cohomology is locally given by equivariant cohomology of loci in $X$ which are fixed by some subgroups of $T$ indexed by points of $E_T$.
\begin{definition}\label{definition:T(x)}
Let $e$ be a point of $E_T$.  Let $S(e)$ be the set of subgroups $K\subset T$ such that the closure of $e$, $\overline{\{e\}}$, belongs to $E_{K}\subset E_T$. Then we define 
$$T(e):=\bigcap_{K\in S(e)}K$$
and 
$$T'(e):=T/T(e)$$
\end{definition}


Let 
$$H_{T}^{\oplus,\ast}(X)=\bigoplus_{i\in\bZ}H_T^{\ast+2i}(X;\bC)$$
be the \emph{sum}-$\bZ_2$-periodization of $T$-equivariant singular cohomology of $X$, with complex coefficients.
This is a module over the even $T$-equivariant cohomology of the point $$H_T^{\oplus,0}(\ast)=\bC[u_1,\dots,u_n]$$ or equivalently a quasi-coherent sheaf over $\Spec H^{\oplus,0}_T(\ast)\simeq\ft_{\bC}\simeq\bA^{n}_{\bC}$, where $\ft_{\bC}$ is the complexified Lie algebra of $T$. Let us call $\mathcal{H}_T(X)$ this quasi-coherent sheaf. 
\begin{remark}
Under our assumptions, $\cH_T(X)$ is a \emph{coherent} sheaf on $\ft_\bC$.
\end{remark}
We denote by $\cH_T^{\an}(X)$ the analytification, i.e. the coherent sheaf 
$$\mathcal{H}_T^{\an}(X)=\mathcal{H}_T(X)\otimes_{\cO_{\ft_\bC}}\cO_{\ft_\bC}^\an$$

There is a quotient map
$$\mathrm{exp^2}:\ft_\bC\to E_T$$
which is an isomorphism if restricted to sufficiently small analytic disks $U$ in $E_T$. Let us call $\mathrm{log^2}$ the local inverse. Moreover, the group structure (we use multiplicative notation) on $E_T$ induces translation maps
\begin{align*}
\tau_e: & E_T\to E_T \\ & f\mapsto fe
\end{align*}
for all closed points $e$ in $E_T$.
Then, for a closed point $e\in E_T$ and a sufficiently small analytic neighbourhood $U_e$ of $e$, we set
$$\El_T^\an(X)|_{U_e}=(\tau_e\circ\mathrm{exp^2})_{\ast}\cH_T^{\an}(X^{T(e)})|_{\mathrm{log^2}(e^{-1}U_e)}$$
The algebraic coherent sheaf obtained from $\El_T^{\an}(X)$ via GAGA is denoted $\El_T(X)$.

The completions of Grojnowski's sheaf over closed points $e$ of $E_T$ can be expressed in terms of the \emph{product}-$\bZ_2$-periodization of $T$-equivariant singular cohomology:
$$\El_T(X)_{\widehat{e}}=H_{T}^{\prod,\ast}(X^{T(e)})\simeq H_{T}^{\oplus,\ast}(X^{T(e)})\otimes_{\cO(\ft_\bC)}\cO_{E_T,\widehat{e}}$$
where $\cO_{E_T,\widehat{e}}$ is a module over $\cO(\ft_\bC)$ via the completed multiplication map $\widehat{\mu}_e:E_{T,\widehat{1}}\to E_{T,\widehat{e}}$ and the identification $E_{T,\widehat{1}}\simeq \ft_{\bC,\widehat{0}}$.

\subsection{Adelic descent}
\label{sec:adelic}
We now review adelic descent theory for $n$-dimensional schemes. This theory was first introduced by Parshin \cite{ParAd} and Beilinson \cite{BeilAd}. A review of this theory can be found in \cite{HubAd} and \cite{MorAd}. Recently, Groechenig \cite{Groech} made a very relevant contribution to the theory. His paper is the main reference for the short reminder that follows. 

Let $X$ be a Noetherian scheme. For two points $x$ and $y$ we say $x\geq y$ if $y\in\overline{\{x\}}$.
We let $|X|_k$ denote the set of \emph{$k$-chains} on $X$, i.e. sequences of $k+1$ ordered points $(x_0\geq\dots\geq x_k)$ in $X$. If $k=0$, we equivalently write $|X|=|X|_0$.
Finally, for a subset $T\subset|X|_k$, we call
$$_{x}T:=\{\Delta\in |X|_{k-1}|(x\geq \Delta)\in T\}$$

We are now ready to define sheaves of ad\`eles on $X$ for a choice of $T\subset|X|_{k}$. The ad\`eles are the unique family of exact functors of abelian categories 
$$\bfA_{X}(T,-):\QCoh(X)\to\Mod_{\cO_X}$$ 
such that:
\begin{itemize}
	\item $\bfA_{X}(T,-)$ commutes with directed colimits;
	\item if $\cF$ is coherent and $k=0$, $\bfA_{X}(T,\cF)=\prod_{x\in T}\lim_{r\geq 0}\tilde{j}_{rx}\cF$;
	\item if $\cF$ is coherent and $k>0$, $\bfA_{X}(T,\cF)=\prod_{x\in|X|}\lim_{r\geq 0}\bfA_{X}(_{x}T,\tilde{j}_{rx}\cF)$.
\end{itemize}

The notation we use is borrowed from \cite{MorAd}. $\tilde{j}_{rx}$ denotes the functor $j_{rx\,\ast}j_{rx}^{\ast}$, where 
$$j_{rx}:\Spec \cO_{X,x}/\mathfrak{m}_x^r\to X$$
is the canonical immersion of an $r$-thickening of the point $x$. $\cO_{X,x}$ is the local ring at $x$ and $\mathfrak{m}_x$ its maximal ideal.

We denote by $\bA_{X}(T,\cF)$ the global sections $\Gamma(X, \bfA_{X}(T,\cF))$, and call them \emph{the groups of ad\`eles}. 

The sets $|X|_k$ admit the structure of a simplicial set by defining face and degeneracy maps by the operations of removing a point from a chain or doubling one up respectively. Let this simplicial set be denoted by $|X|_{\bullet}$. This implies that the sheaves of ad\`eles assemble into a cosimplicial sheaf of $\cO_{X}$-modules $\bfA_{X}(T_\bullet,\cF)$, for some $T_\bullet\subset|X|_\bullet$. In the case $T_\bullet=|X|_\bullet$, this cosimplicial sheaf is denoted by $\bfA_{X}^{\bullet}(\cF)$, and its global sections by $\bA_{X}^{\bullet}(\cF)$. If $\cF=\cO_X$, the notation we reserve is $\bfA_{X}^{\bullet}$ and $\bA_{X}^{\bullet}$ respectively.

It is possible to consider the cosimplicial sheaf of ``products of local ad\`eles''
$$[n]\mapsto\prod_{\Delta\in|X|_n}\bfA_{X}(\Delta,\cF)$$
i.e. we choose $T=\{\Delta\}$ and then take the product over all chains $\Delta$.
The content of Theorem 2.4.1 in \cite{HubAd} is that the natural inclusion of the ad\`eles into this product respects the cosimplicial structures.

Adelic descent theory allows to reconstruct sheaves from their \emph{adelic descent data}, i.e. the data of the sheaves of ad\`eles or their global sections. We state two theorems on adelic descent for Noetherian n-dimensional schemes. The first one is due to Groechenig and holds for perfect complexes in the context of small $\infty$-categories.

\begin{theorem}[Theorem 3.1 in \cite{Groech}]
Let $X$ be a Noetherian scheme. Then adelic reconstruction is an equivalence of symmetric monoidal $\infty$-categories
$$\Perf^{\otimes}(X)\simeq\Tot\Perf^{\otimes}(\bA_{X})$$
\end{theorem}

The following theorem due to Beilinson appears as Theorem 1.16 in \cite{Groech}. This is a classical theorem for the abelian category of quasi-coherent sheaves.
\begin{theorem}[Beilinson \cite{BeilAd}]
Let $\cF$ be a quasi-coherent sheaf on $X$. The augmentation $\cF\to\bfA_{X}^{\bullet}(\cF)$ is a resolution of $\cF$ by flasque $\cO_X$-modules. In particular, the totalization of the ad\`eles $\Tot \bA_{X}^{\bullet}(\cF)$ computes the cohomology of $\cF$.
\end{theorem}

\begin{remark}
The arguments made by Groechenig in \cite{Groech} hold without variations in the context of $\bZ_2$-periodic sheaves. This is the context we are interested in, as we are dealing with $\bZ_2$-periodic cohomology theories.
\end{remark}

\subsection{Shifted tangent bundles and loop spaces}\label{ssec:stbdls}
In this subsection we recall   two fundamental objects in derived algebraic geometry that will be used in the second part of this paper. For general references on derived algebraic geometry in the context of $E_\infty$-rings, see Lurie's work \cite{HTT}, \cite{HA} and \cite{SAG}. In the context of simplicial commutative rings/cdgas the theory has been developed by To\"en and Vezzosi in \cite{HAGI} and \cite{HAGII}. For a short review, see the preliminaries section of \cite{MapStI}.

Let $X$ be a derived stack, and $\bL_X$ be its cotangent complex.
\begin{definition}
The \emph{shifted tangent stack} of $X$ is the derived stack 
$$T_{X}[-1]:=\Spec_{\cO_X}\Sym\,\bL_X[1]$$
\end{definition}
\begin{definition}
The \emph{derived loop space} of $X$ is the stack
$$\cL X:=\Map{S^1}{X}$$
\end{definition}
In the above, $\Map{-}{-}$ denotes the mapping stack as derived stacks.
\begin{remark}
Since there is an equivalence $S^1\simeq \ast\coprod_{\ast\coprod\ast}\ast$, we have that 
$$\cL X\simeq X\times_{X\times X}X$$
\end{remark}
The algebra of (derived) global sections of the structure sheaf of $S^1$ is formal over a field of characteristic zero $k$. In particular, over $k$ this algebra is isomorphic to $k[\epsilon]$, where $\epsilon$ is a variable in cohomological degree one. In particular, the cosimplicial spectrum in the sense of \cite{ToenPhD}, $\Spec k[\epsilon]=\Spec \cO(S^1)$, is the \emph{affinization} of the circle $S^1$ (see also \cite{MRT} for more on affinization and affine stacks). This allows us to introduce a third object:
\begin{definition}
The \emph{unipotent loop space} of $X$ is the stack
$$\cL^u X:=\Map{\Spec k[\epsilon]}{X}$$
\end{definition}

If $X$ is a scheme over $k$, Ben-Zvi and Nadler establish a Zariski codescent result for the loop space, Lemma 4.2 in \cite{BZNLoopConn}, which has the consequence that for schemes the three objects described above coincide (Proposition 4.4 in \cite{BZNLoopConn}), as they do for affine schemes. In particular, taking global sections gives a form of the HKR theorem, as the global sections of the loop space compute Hochschild homology and the global sections of the shifted tangent bundle compute the de Rham complex of $X$.

We will be interested in the shifted tangent bundle and in the derived loop space of quotient stacks. In the following few lines, we review the case of classifying stacks. 
\begin{example}
Let $G$ be a smooth affine reductive algebraic group over $k$. The shifted tangent bundle of $BG$, $T_{BG}[-1]$, is given by 
$$T_{BG}[-1]\simeq [\fg/G]$$
where $\fg$ is the Lie algebra of $G$ and $G$ acts on $\fg$ via the adjoint representation. In particular, the affinization of the shifted tangent is given by the GIT quotient, $\fg//G$.
This example can be found in \cite{BZNLoopConn}. It depends on the fact that the cotangent complex of $BG$, pulled back along $a:\Spec k\to BG$, is given by:
$$a^\ast\bL_{[\Spec k/G]}\simeq \fg^{\vee}[-1]$$
where $\fg^{\vee}$ is the dual to the Lie algebra $\fg$.
The shift places $\fg^{\vee}$ in degree zero, so that in the end 
$$\Sym\,\bL_{BG}[1]\simeq\Sym(\fg^{\vee})^{G}$$
\end{example}

\begin{example}
Let $G$ be a smooth affine reductive algebraic group over $k$. The derived loop space of $BG$, $\cL BG$, is given by 
$$\cL BG\simeq [G/G]$$
where $G$ acts on itself via the adjoint action. In particular, the affinization of the derived loop space is given by the GIT quotient, $G // G$. This follows because the derived loop space of $BG$ classifies $G$-local systems on the circle, which are given exactly by $[G/G]$. For more details, see \cite{BZNLoopConn}.
\end{example}

\begin{remark}
Let $G$ be a compact Lie group.
By work of Atiyah--Bott \cite{AtiBott} we have:
$$\Spec H_G^{\oplus,0}(\ast;\bC)\simeq \fg_\bC//G_\bC$$
where $\fg_\bC$ and $G_\bC$ denote the complexifications of the Lie algebra of $G$ and of $G$ itself respectively. 
Similarly, we have
$$\Spec K_G^0(\ast)\otimes_{\bZ}\bC\simeq G_\bC//G_\bC$$
In particular, over the complex numbers, 
$$\Aff{T_{BG}[-1]}\simeq\Spec H_{G^c}^{\oplus,0}(\ast)$$
$$\Aff{\cL BG}\simeq\Spec K_{G^c}^0(\ast)$$
Here, $G^c$ is the maximal compact subgroup of $G$ ($G$ is the complexification of $G^c$).
\end{remark}

\section{The Adelic Decomposition of Equivariant Elliptic Cohomology}\label{section:adelicGroj}
In this section we define $k$-rationalized $T$-equivariant elliptic cohomology, where $k$ is a field of characteristic zero, for finite $T$-CW complexes. The main point in the construction is the localization theorem, which allows us to describe this object inductively via its adelic descent data. When the torus is $S^1$, the adelic descent data for $k$-rationalized equivariant elliptic cohomology can be described in terms of Borel equivariant singular cohomology with coefficients in $k$.
\subsection{The rank one case}
We begin with a definition of $k$-rationalized $S^1$-equivariant elliptic cohomology of finite $T$-CW-complexes. 

\begin{definition}\label{def:ell1}
Let $k$ be a field of characteristic zero and $E$ be an elliptic curve over $k$, and let $X$ be a finite $T=S^1$-CW-complex. We define $k$-rationalized $T$-equivariant elliptic cohomology as the coherent sheaf $\El_T(X)$ on $E$ such that:
\begin{itemize}
\item For a closed point $e\in E$, $\El_T(X;k)_{\widehat{e}}=H_{T}^{\oplus, \ast}(X^{T(e)};k)\otimes_{\cO(\ft)}\cO_{E,\widehat{e}}$;
\item $\El_T(X;k)_{\widehat{\eta}}=H^{\oplus, \ast}(X^T;k)\otimes_{k}\cO_{E,\eta}$;
\item $\El_T(X;k)_{\widehat{(\eta>e)}}=H^{\oplus, \ast}(X^T;k)\otimes_{k}\Frac\cO_{E,\widehat{e}}$.
\end{itemize}
The (reduced) cosimplicial structure is induced by the cosimplicial structure on the ad\`eles for $E$ and pullback and change of group maps in equivariant singular cohomology:
\begin{itemize}
\item the map $\El_T(X;k)_{\widehat{\eta}}\to\prod_{e\in |E|_{cl}}\El_T(X;k)_{\widehat{(\eta>e)}}$ is given by the identity tensored with the coface map $\cO_{E,\eta}\to\Frac\cO_{E,\widehat{e}}$ of the ad\`eles of $\cO_E$;
\item the map $\prod_{e\in |E|_{cl}}\El_T(X;k)_{\widehat{e}}\to\prod_{e\in |E|_{cl}}\El_T(X;k)_{\widehat{(\eta>e)}}$ is given by the product of tensor products of pullback maps in cohomology along inclusions $X^{T}\hookrightarrow X^{T(e)}$ and coface maps $\cO_{E,\widehat{e}}\to\Frac\cO_{E,\widehat{e}}$ of the ad\`eles of $\cO_E$.
\end{itemize}
\end{definition}
In the above definition, $|E|_{cl}$ denotes the set of closed points of $E$. 

\begin{remark}
In the above definition we only describe the \emph{reduced} cosimplicial structure of the ad\`eles, and not the full cosimplicial structure. This is enough. See for example \cite{MorAd} for a description of the reduced ad\`eles in the one dimensional case, and of the relevant (reduced) cosimplicial structure.
\end{remark}

\begin{proposition}
Let $X$ be a finite $T=S^1$-CW-complex. Let $\El_T(X)$ be Grojnowski's $T$-equivariant elliptic cohomology. The adelic descent data for $\El_T(X)$ is 
\begin{itemize}
\item For a closed point $e\in E$, $\El_T(X)_{\widehat{e}}=H_{T}^{\oplus, \ast}(X^{T(e)};\bC)\otimes_{\cO(\ft)}\cO_{E,\widehat{e}}$;
\item $\El_T(X)_{\widehat{\eta}}=H^{\oplus, \ast}(X^T;\bC)\otimes_{k}\cO_{E,\eta}$;
\item $\El_T(X)_{\widehat{(\eta>e)}}=H^{\oplus, \ast}(X^T;\bC)\otimes_{k}\Frac\cO_{E,\widehat{e}}$.
\end{itemize}
The reduced cosimplicial structure is given by the coface maps for the ad\`eles of $\cO_E$ and pullback maps in singular cohomology.
\end{proposition}
\begin{proof}
The case of closed points is explained in \cite{Gan14}. For the generic point, the localization theorem yields
$$\El_T(X)_{\widehat{\eta}}\simeq \left (c_{\eta}^{\ast}H^{\oplus,\ast}(X^{T(\eta)};\bC)\right )_{\widehat{\eta}}\simeq H^{\oplus, \ast}(X^T;\bC)\otimes_{k}\cO_{E,\eta}$$
For the chain $(\eta>e)$, by definition we have
$$\El_T(X)_{\widehat{(\eta>e)}}=\bA_{E}((e), \tilde{j}_{\eta}\El_T(X))$$
hence by localization
$$\bA_{E}((e), \tilde{j}_{\eta}\El_T(X))\simeq H^{\oplus, \ast}(X^T;\bC)\otimes_{k}\Frac\cO_{E,\widehat{e}}$$

We now describe the cosimplicial structure. The map
$$\bA_{E}((\eta),\El_{S^1}(X))\to \bA_{E}((\eta>e),\El_{S^1}(X))$$
is given by the identity on $H^{\oplus,\ast}(X^T)$ tensored with the coface map relative to the ad\`eles for the structure sheaf of $E$. The map
$$\bA_{E}((e),\El_{S^1}(X))\to \bA_{E}((\eta>e),\El_{S^1}(X))$$
is given by the pullback along the inclusion $X^{T(e)}\hookrightarrow X$ in singular cohomology tensored with the coface map as in the case above.
\end{proof}

The above proposition shows that, if $k=\bC$, Definition \ref{def:ell1} recovers $S^1$-equivariant elliptic cohomology in the sense of Grojnowski:
\begin{corollary}
Let $X$ be a finite $S^1$-CW-complex. Then
$$\El_{S^1}(X;\bC)\simeq\El_{S^1}(X)$$
\end{corollary}

We now construct pullback and change of group maps in $S^1$-equivariant elliptic cohomology with coefficients in $k$.

\begin{lemma}\label{lemma:pullbacks1}
A $T$-equivariant map $f:X\to Y$ induces a pullback map
$$f^\ast:\El_{S^1}(Y;k)\to\El_{S^1}(X;k)$$
\end{lemma}
\begin{proof}
We can give the adelic decomposition of $f^\ast$ as follows. For a reduced chain $\Delta$ on $E$, we declare
$$\bA_{E}(\Delta, f^{\ast}):H^{\oplus,\ast}_{T}(Y^{T(\Delta)};k)\otimes_{k}\cO_{E,\eta}\to H^{\oplus,\ast}_{T}(X^{T(\Delta)};k)\otimes_{k}\cO_{E,\eta}$$
to be the tensor product of the pullback map $H^{\oplus,\ast}_{T}(Y^{T(\Delta)};k)\to H^{\oplus,\ast}_{T}(X^{T(\Delta)};k)$ in equivariant cohomology relative to the map $f^{T(\Delta)}: X^{T(\Delta)}\to Y^{T(\Delta)}$. Here we denote by $T(\Delta)$ the subgroup $T(x)$, where $x$ is the maximal point in $\Delta$.

The coface maps in Definition \ref{def:ell1} are defined to be the tensor product of pullback maps in singular cohomology and the coface maps for the ad\`eles for $E$. In particular, the maps $\bA_{E}(\Delta, f^{\ast})$ assemble into a cosimplicial map
$$\bA_{E}^{\bullet}(f^{\ast}):\bA_{E}^{\bullet}(\El_{S^1}(Y;k))\to\bA_{E}^{\bullet}(\El_{S^1}(X;k))$$
as pullbacks commute with pullbacks.
\end{proof}

\begin{lemma}\label{lemma:grpch11}
Let $X$ be a finite $S^1$-CW-complex, and $1:\ast\to S^1$ be the identity of $S^1$. The map $1$ induces a map of sheaves on $E$
$$\El_{S^1}(X;k)\to (1_E)_{\ast} H^{\oplus,\ast}(X;k)$$
where $1_E:E_1\simeq\Spec k\to E\simeq E_{S^1}$ is the identity section of $E$.
\end{lemma}
\begin{proof}
The adelic decomposition of this map can be described as follows. Over the identity section $1_E$, map is induced by the analogous map at the level of equivariant cohomology:
$$H_{S^1}^{\oplus,\ast}(X)\to H^{\oplus,\ast}(X)$$
while everywhere else it is given by the zero map.
\end{proof}

\begin{lemma}\label{lemma:grpchS1}
Let $S^1\to S^1/S^1\simeq\ast$ be the quotient map, and $X$ be a finite CW-complex. This map induces a map of sheaves on $E$
$$H^{\oplus,\ast}(X;k)\to c_{\ast} \El_{S^1}(X;k)$$
where $c:E\to\Spec k$ is the structure map of $E$. In the above, $S^1$ acts on $X$ through the homomorphism $S^1\to S^1/S^1$, i.e. trivially.
\end{lemma}
\begin{proof}
By adjunction, we can equivalently describe the map 
$$c^{\ast} H^{\oplus,\ast}(X;k)\to\El_{S^1}(X;k)$$ 
This map is actually an isomorphism. Indeed, as $S^1$ acts on $X$ via the homomorphism $S^1\to S^1/S^1$, the equivariant elliptic cohomology $\El_{S^1}(X;k)$ is the tensor product 
$$H^{\oplus,\ast}(X;k)\otimes_{k}\cO_E$$
which is exactly the pullback of $H^{\oplus,\ast}(X;k)$ along the structure map $c:E\to\Spec k$.
\end{proof}

\subsection{The higher rank case}
We define $\El_{T}(X;k)$ for higher rank tori via induction on the rank of the torus $T$. This allows us to avoid describing the ad\`eles with respect to a chain $\Delta=(x,x_1,\dots,x_p)$ in terms of singular cohomology whenever $x_p$ is not closed. If $x_p$ is closed it is possible to give such a description. 

\begin{definition}\label{def:elln}
Let $k$ be a field of characteristic zero and $E$ be an elliptic curve over $k$, $T$ be a real torus of rank $n$ and $X$ be a finite $T$-CW-complex. The $k$-rationalized $T$-equivariant elliptic cohomology of $X$ is the coherent sheaf $\El_T(X;k)$ on $E_T$ inductively defined by the following adelic descent data:
\begin{itemize}
\item given a reduced chain $\Delta=(x,x_1,\dots,x_p)$,\\
$\bfA_{E_T}(\Delta, \El_T(X;k))=c_{x}^{\ast}\El_{T'(x)}(X^{T(x)};k)\otimes_{\cO_{E_T}}\cO_{E_T,\widehat{\Delta}}$;
\item given a reduced chain $\Delta=(x,x_1,\dots,x_p)$, if $x_p$ is closed,\\
$\bfA_{E_T}(\Delta, \El_T(X;k))=H_{T}^{\oplus, \ast}(X^{T(x)};k)\otimes_{\cO(\ft)}\cO_{E_T,\widehat{\Delta}}$.
\end{itemize}
Here $\El_{T'(x)}(X^{T(x)};k)$ is $k$-rationalized $T'(x)$-equivariant elliptic cohomology of $X^{T(x)}$ and 
$$c_x:E_T\to E_{T'(x)}$$
The cosimplicial structure is the following:
\begin{itemize}
\item if the chain $\Delta=(x,x_1,\dots,x_p)$ is such that $x_p$ is closed we use Proposition 3.2.2 in \cite{HubAd}, so that the relevant coface maps for removing a point of the chain are given by tensor products of the corresponding coface maps for the ad\`eles for the structure sheaf of $E_T$ and pullback maps in equivariant cohomology;
\item If $x_p$ is not closed, the coface maps are tensor products of the corresponding maps of the ad\`eles of $\cO_{E_T}$ and pullback and change of group maps in equivariant elliptic cohomology with respect to tori of rank strictly smaller than $\mathrm{rk}(T)$. 
\end{itemize}
\end{definition}
\begin{remark}
As in Definition \ref{def:ell1}, we give only the reduced cosimplicial structure and deal with reduced chains. Moreover, as $x$ is not a closed point, the torus $T'(x)$ is of rank strictly smaller than the rank of $T$. 
\end{remark}
\begin{remark}
If we remove a point $x_i$, $i\in\{1,\dots,p-1\}$, the pullback map in equivariant cohomology is the identity, since the point $x$ is not removed. This map differs from the identity only when we remove $x$ from $\Delta$. 
Similarly, if we remove $x_p$ from $\Delta$, we need to use the presentation in terms of equivariant elliptic cohomology for smaller rank tori when describing the associated coface map, as $x_{p-1}$ is not closed.
\end{remark}

To ensure that Definition \ref{def:elln} is well posed, we need to produce pullback and restriction maps inductively.
\begin{lemma}\label{lemma:pullbacksn}
Let $f:X\to Y$ be a $T$-equivariant map, where $T$ is a torus of rank $n$. Assume that $k$-rationalized equivariant elliptic cohomology has pullbacks with respect to $\tilde{T}$-equivariant maps, where $\tilde{T}$ is any torus of rank strictly smaller than $n$. Moreover, assume that change of group maps with respect to maps of tori exist in $k$-rationalized $\tilde{T}$-equivariant elliptic cohomology. 
Then $f$ induces a pullback map
$$f^\ast:\El_{T}(Y;k)\to\El_{T}(X;k)$$
in $k$-rationalized $T$-equivariant elliptic cohomology.
\end{lemma}
\begin{proof}
We construct these maps using the adelic decomposition described in Definition \ref{def:elln}. In the case $x_p$ is not closed, we use pullback maps for equivariant elliptic cohomology relative to smaller rank tori, which exist by the inductive hypothesis. Indeed, the maps we need to construct is 
$$\bfA_{E_T}(\Delta,\El_{T}(Y;k))\to\bfA_{E_T}(\Delta,\El_{T}(X;k))$$
for all reduced chains $\Delta=(x,x_1,\dots,x_p)$, respecting the reduced cosimplicial structure. The map above is constructed then as the tensor product
$$c_{x}^{\ast}\El_{T'(x)}(X^{T(x)};k)\otimes_{\cO_{E_T}}\cO_{E_T,\widehat{\Delta}}\to c_{x}^{\ast}\El_{T'(x)}(Y^{T(x)};k)\otimes_{\cO_{E_T}}\cO_{E_T,\widehat{\Delta}}$$
of the pullback along $f^{T(x)}:X^{T(x)}\to Y^{T(x)}$ with the identity of $\cO_{E_T,\widehat{\Delta}}$.
If $x_p$ is closed, we equivalently use the presentation of the ad\`eles in terms of equivariant cohomology. In this case, the map between the ad\`eles becomes the tensor product
$$H_{T}^{\oplus, \ast}(X^{T(x)};k)\otimes_{\cO(\ft)}\cO_{E_T,\widehat{\Delta}}\to H_{T}^{\oplus, \ast}(Y^{T(x)};k)\otimes_{\cO(\ft)}\cO_{E_T,\widehat{\Delta}}$$
of the pullback in equivariant cohomology along $f^{T(x)}:X^{T(x)}\to Y^{T(x)}$ with the identity of $\cO_{E_T,\widehat{\Delta}}$.
The commutativity with the cosimplicial structure is clear, as pullbacks commute with pullbacks.
\end{proof}

The proof of the next lemma works very similarly to that of the previous lemma.
\begin{lemma}\label{lemma:gpchangen}
Let $T_1\to T_2$ be a group homomorphism, and $T_1$ and $T_2$ be tori of rank smaller or equal to $n$. Assume that $k$-rationalized equivariant elliptic cohomology has pullbacks with respect to $\tilde{T}$-equivariant maps, where $\tilde{T}$ is any torus of rank strictly smaller than $n$. Moreover, assume that change of group maps with respect to maps of tori exist in $k$-rationalized $\tilde{T}$-equivariant elliptic cohomology.
Then we have a change of group map
$$\El_{T_2}(Y;k)\to c_{1,2 \ast}\El_{T_1}(X;k)$$
In the above, $c_{1,2}$ is the projection map
$$E_{T_1}\to E_{T_2}$$
\end{lemma}
\begin{proof}
As in the proof of Lemma \ref{lemma:pullbacksn}, the relevant map between the ad\`eles relative to the reeduced chain $\Delta$ is given by a tensor product of the identity of the ad\`ele $\cO_{E_T,\widehat{\Delta}}$ with the change of group map in either $k$-rationalized equivariant elliptic cohomology with respect to tori of smaller rank if $x_p$ is not closed or in singular cohomology for $x_p$ closed.
\end{proof}

As a consequence of the two lemmas \ref{lemma:pullbacksn} and \ref{lemma:gpchangen}, we obtain that $k$-rationalized $T$-equivariant elliptic cohomology has pullbacks with respect to $T$-equivariant maps and change of group maps with respect to homomorphisms of tori.

\begin{corollary}\label{corollary:pullchange}
Let $f:X\to Y$ be a $T$-equivariant map. Then $f$ induces a pullback map
$$f^{\ast}:\El_{T}(Y;k)\to\El_{T}(X;k)$$
Let $T_1\to T_2$ be a group homomorphism, and $T_1$ and $T_2$ be tori. We have an induced change of group map
$$\El_{T_2}(Y;k)\to c_{1,2 \ast}\El_{T_1}(X;k)$$
where $c_{1,2}$ is the projection map
$$E_{T_1}\to E_{T_2}$$
\end{corollary}

We now observe that, if $k=\bC$, our definition recovers Grojnowski's. To do so, we need to compute the ad\`eles of Grojnowski's sheaf first. 

\begin{lemma}\label{lemma:localizadeles}
Given a chain $\Delta=(x,x_1,\dots,x_p)$ we have that 
$$\bfA_{E_T}(\Delta, \El_T(X)) \simeq c_{x}^{\ast}\El_{T'(x)}(X^{T(x)})\otimes_{\cO_{E_T}}\cO_{E_T,\widehat{\Delta}}$$
\end{lemma}
\begin{proof}
By definition, 
$$\bfA_{E_T}(\Delta, \El_T(X))=\lim_{r\geq 0}\bfA_{E_T}(_{x}\Delta, \tilde{j}_{rx}\El_T(X))$$
The localization theorem dictates that 
$$\tilde{j}_{rx}\El_T(X)\simeq\tilde{j}_{rx}\El_T(X^{T(x)})\simeq\tilde{j}_{rx}c_x^\ast\El_{T'(x)}(X^{T(x)})$$
hence 
$$\bfA_{E_T}(\Delta, \El_T(X))\simeq\lim_{r\geq 0}\bfA_{E_T}(_{x}\Delta, \tilde{j}_{rx}c_x^\ast\El_{T'(x)}(X^{T(x)}))$$
$$\simeq\bfA_{E_T}(\Delta, c_x^\ast\El_{T'(x)}(X^{T(x)}))$$
\end{proof}

Whenever the chain $\Delta=(x,x_1,\dots,x_p)$ is such that $x_p$ is closed, it is easy to obtain a description of the ad\`eles in terms of singular cohomology:
\begin{lemma}\label{lemma:localizadelesclosed}
Given a chain $\Delta=(x,x_1,\dots,x_p)$ with $x_p$ closed, we have that 
$$\bfA_{E_T}(\Delta, \El_T(X))=H_{T}^{\oplus,\ast}(X^{T(x)})\otimes_{\cO(\ft)}\cO_{E_T,\widehat{\Delta}}$$
\end{lemma}
\begin{proof}
This is a simple application of Proposition 3.2.1 from \cite{HubAd}, together with the description of the ad\`eles at closed points. The $C_{x_i}$ operations do not affect the singular cohomology modules as they are finitely presented over $H_T^{\oplus}$, since this ring is Noetherian and $X$ is a finite $T$-CW-complex. We can replace the equivariant cohomology of the locus fixed by $T(x)$ with the equivariant cohomology of the locus fixed by $T(x_p)$ by applying the localization formula.
\end{proof}

The two Lemmas \ref{lemma:localizadeles} and \ref{lemma:localizadelesclosed} allow us to deduce immediately the following Theorem:
\begin{theorem}\label{thm:equivGroj}
Let $X$ be a finite $T$-CW-complex. Then
$$\El_{T}(X;\bC)\simeq\El_{T}(X)$$
\end{theorem}
\begin{proof}
Lemmas \ref{lemma:localizadeles} and \ref{lemma:localizadelesclosed} give us a description of the ad\`eles of Grojnowski's equivariant elliptic cohomology relative to reduced chains $\Delta$. The cosimplicial structure corresponds exactly to the one described in Definition \ref{def:elln}, hence we obtain an isomorphism of the ``product of local ad\`eles" 
$$\prod_{\Delta\in|E_T|_\bullet}\bfA_{E_T}(\Delta,\El_{T}(X;\bC))\simeq\prod_{\Delta\in|E_T|_\bullet}\bfA_{E_T}(\Delta,\El_{T}(X))$$ 
As the ad\`eles embed in the product of local ad\`eles as a cosimplicial subcomplex, we obtain the isomorphism
$$\El_{T}(X;\bC)\simeq\El_{T}(X)$$
\end{proof}

\subsection{Extension to compact Lie groups}
In this subsection we extend the construction from tori to compact Lie groups. There is a canonical way to do so, and is explained in \cite{Gan14}. Let $G$ be a compact Lie group, and call $T$ its maximal torus and $W$ the Weyl group. Then call 
$$E_G=E_T/W$$
\begin{definition}\label{def:ellnG}
The $k$-rationalized $G$-equivariant elliptic cohomology of a finite $G$-CW-complex $X$ is the coherent sheaf of $W$-invariants on $E_G$:
$$\El_{G}(X;k) := \El_T(X;k)^{W}$$
\end{definition}

\subsection{Cochain-level variant}\label{ssec:chains}
The same construction we developed in this section can be adapted to a cochain-level variant. The key to this construction is the formality of the algebra of cochains on the Borel quotient $\ast//T$. Let 
$$C^{\oplus,\ast}_T(X):=C^{\oplus,\ast}(X//T)$$
denote the sum-$\bZ_2$-periodization of the cochains on the Borel construction $X//T$. The algebra $C^{\oplus,0}_T(\ast)$ is formal, in particular $C^{\oplus,0}_T(\ast)\simeq\cO(\ft)$. Hence we can redefine the adelic descent data given in Definition \ref{def:elln} by replacing the cohomology of $X^{T(x)}//T$ with its cochains. The same inductive procedure yields, by an application of Groechenig's Theorem 3.1 in \cite{Groech}, an object in $\Perf(E_T)$ (as $X$ is a finite $T$-CW-complex), that we denote by 
$$\cC\El_T(X)$$
If the base field $k$ is taken to be the field of the complex numbers, this object can be equivalently constructed following Grojnowski's methodology. See the Preliminaries in \cite{MapStI} for the details. 

\section{The Adelic Decomposition for Equivariant Cohomology and K-theory}\label{section:adelicsingK}
In this section we study the adelic description of equivariant singular cohomology and equivariant K-theory, summarized in the following two lemmas. These lemmas will play a role in section \ref{section:geometry}, where we construct a periodic cyclic homology model for these two cohomology theories.

\subsection{Equivariant cohomology}
We work over a field $k$ of characteristic zero. Throughout this section, we denote by
$$\cH_{T}(X)$$
the quasi-coherent sheaf on
$$\ft^{alg}=\bA^1_{k}\otimes_{\bZ}\check{T}$$
whose global sections are
$$\Gamma(\ft^{alg},\cH_{T}^{\oplus,\ast}(X))=H_{T}(X;k)$$
The adelic decomposition of singular cohomology is obtained as a simple application of the localization formula and the observation stated in Groechenig's \cite{Groech} after Remark 1.9: if $X$ is an affine Noetherian scheme and $\cF$ a quasi-coherent sheaf, $\bA_{X}^{n}(\cF)\simeq \cF(X)\otimes_{\cO_{X}(X)}\bA_{X}^{n}$. 

First we need to introduce the notion of subgroup of $T$ associated to an element of the Lie algebra $\ft$.

\begin{definition}\label{def:dictH}
Let $\xi\in\ft^{alg}$ be a point of $\ft^{alg}$. We define 
$$T(\xi)\leq T$$
to be the smallest subgroup of $T$ such that the closure of the point $\xi$, $\overline{\{\xi\}}$, is contained in 
$$\ft^{alg}_{T(\xi)}:=\bA^1_k\otimes_{\bZ}\check{T}(\xi)$$
where $\check{T}(\xi)$ is the cocharacter lattice of $T(\xi)$.
\end{definition}

\begin{lemma}\label{lemma:AdSC}
Let $X$ be a finite $T$-CW-complex and let $\Delta=(\xi>\xi_1>\dots>\xi_k)$ be a chain of length $k$ on $\ft_{\bC}$. Then
$$\bA_{\ft^{alg}}\left (\Delta, \cH_{T}(X)\right )\simeq H_{T}^{\oplus}(X^{T(\xi)})\otimes_{\cO(\ft^{alg})}\bA_{\ft^{alg}}\left (\Delta,\cO_{\ft^{alg}}\right )$$
\end{lemma}
\begin{proof}
Groechenig's observation yields
$$\bA_{\ft^{alg}}\left (\Delta, \cH_{T}(X)\right )\simeq H_{T}^{\oplus}(X)\otimes_{\cO(\ft^{alg})}\bA_{\ft^{alg}}\left (\Delta,\cO_{\ft^{alg}}\right )$$
and the localization theorem implies that the pullback along the inclusion $X^{T(\xi)}\hookrightarrow X$ becomes an isomorphism after tensoring with the ad\`eles $\bA_{\ft^{alg}}(\Delta,\cO_{\ft^{alg}})$ over the base ring $\cO(\ft^{alg})$.
\end{proof}

\subsection{Equivariant K-theory}\label{ssec:eqvKThy}
The adelic decomposition of equivariant K-theory can be proved as in the singular cohomology case, as long as the chain starts from a closed point.
We work over a field $k$ of characteristic zero. In particular, we denote by $$\cK_{T}(X)$$ the quasi-coherent sheaf on 
$$T^{alg}:=\bG_{m,k}\otimes_{\bZ}\check{T}$$
with global sections given by
$$\Gamma(T^{alg},\cK_{T}(X))=K_{T}^{\ast}(X)\otimes_{\bZ}k$$
Similarly to the equivariant cohomology case, we need to introduce the notion of subgroup of $T$ associated to an element of $T$.

\begin{definition}\label{def:dictK}
Let $x\in T^{alg}$ be a point of $T^{alg}$. We define 
$$T(x)\leq T$$
to be the smallest subgroup of $T$ such that the closure $\overline{\{x\}}$ of the point $x$ is contained in 
$$T(x)^{alg}=\bG_m\otimes_{\bZ}\widecheck{T(x)}$$
\end{definition}
\begin{lemma}\label{lemma:AdK}
Let $X$ be a finite $T$-CW-complex and let $\Delta=(x>x_1>\dots>x_k)$ be a chain of length $k$ on $T_{\bC}$, such that $x_k$ is a closed point. Then
$$\bA_{T_{\bC}}\left (\Delta, \cK_{T}(X)\right )\simeq H_{T}^{\oplus}(X^{T(x)})\otimes_{\cO(\ft)}\bA_{T_{\bC}}\left (\Delta,\cO_{T_{\bC}}\right )$$
\end{lemma}
\begin{proof}
We apply Proposition 3.2.1 in \cite{HubAd} to the formula computing the completion of equivariant K-theory at a closed point:
$$\cK_{T}(X)_{\widehat{x_k}}\simeq H_T^{\prod,\ast}(X^{T(x_k)})\simeq H_T^{\oplus,\ast}(X^{T(x_k)})\otimes_{\cO(\ft)}\cO_{T,\widehat{x_k}}$$
where we identified the completions $\cO_{\ft,\widehat{0}}\simeq\cO_{T,\widehat{1}}\simeq \cO_{T,\widehat{x_k}}$ (the former isomorphism is the exponential while the latter follows from the group structure on $T$). The cohomology groups $H_T(X^{T(x)})$ are finitely presented over $H(BT)$ since $X$ is a finite $T$-CW-complex.
The final statement follows from the localization formula.
\end{proof}

We observe at this point that Lemma \ref{lemma:AdK} can be used as input data for an inductive construction akin to the one explained in Section \ref{section:adelicGroj} in the elliptic case. All the proofs of Lemmas \ref{lemma:pullbacks1}, \ref{lemma:grpch11}, \ref{lemma:grpchS1}, \ref{lemma:pullbacksn}, \ref{lemma:gpchangen} and Corollary \ref{corollary:pullchange} go through almost unchanged --- the role of $E$ is now played by $\bG_m$ --- thus the same exact construction produces $k$-rationalized equivariant K-theory out of singular cohomology and induction. 
Indeed, we could propose the following 
\begin{definition}\label{def:Kn}
Let $k$ be a field of characteristic $0$, $T$ be a real torus of rank $n$ and $X$ be a finite $T$-CW-complex. The \emph{adelic} $k$-rationalized $T$-equivariant K-theory of $X$
$$\cK_{T}^{ad}(X;k)$$
is the coherent sheaf on $T^{alg}$ inductively defined by the following adelic descent data:
\begin{itemize}
\item given a reduced chain $\Delta=(x,x_1,\dots,x_p)$,\\
$\bfA_{T^{alg}}(\Delta, \cK_T^{ad}(X;k))=c_{x}^{\ast}\cK_{T'(x)}^{ad}(X^{T(x)};k)\otimes_{\cO_{T^{alg}}}\cO_{T^{alg},\widehat{\Delta}}$;
\item given a reduced chain $\Delta=(x,x_1,\dots,x_p)$, if $x_p$ is closed,\\
$\bfA_{T^{alg}}(\Delta, \cK_T^{ad}(X;k))=H_{T}^{\oplus, \ast}(X^{T(x)};k)\otimes_{\cO(\ft)}\cO_{T^{alg},\widehat{\Delta}}$.
\end{itemize}
where 
$$c_x:T\to T'(x)$$
The cosimplicial structure is the following:
\begin{itemize}
\item if the chain $\Delta=(x,x_1,\dots,x_p)$ is such that $x_p$ is closed, the relevant coface maps for removing a point of the chain are given by tensor products of the corresponding coface maps for the ad\`eles for the structure sheaf of $T^{alg}$ and pullback maps in equivariant cohomology;
\item If $x_p$ is not closed, the coface maps are tensor products of the corresponding maps of the ad\`eles of $\cO_{T^{alg}}$ and pullback and change of group maps in adelic $k$-rationalized equivariant K-theory with respect to tori of rank strictly smaller than $\mathrm{rk}(T)$. 
\end{itemize}
\end{definition}
This definition is well-posed, as the rank one case can be treated as in the case of elliptic cohomology and all the lemmas needed hold, as discussed above. Then Lemma \ref{lemma:AdK} proves the following
\begin{theorem}\label{thm:equivRosu}
Let $X$ be a finite $T$-CW-complex, for $T$ a real torus. Then 
$$\cK_{T}^{ad}(X;k)\simeq\cK_{T}(X)$$
as $\bZ_2$-periodic sheaves on $T^{alg}$.
\end{theorem}
\begin{proof}
This proof goes the same as the proof of Theorem \ref{thm:equivGroj}. Lemma \ref{lemma:AdK} gives an equivalence 
$$\prod_{\Delta\in|T^{alg}|_\bullet}\bfA_{T^{alg}}(\Delta,\cK_{T}(X;k))\simeq\prod_{\Delta\in|T^{alg}|_\bullet}\bfA_{T^{alg}}(\Delta,\cK_{T}(X))$$ 
which gives the desired isomorphism.
\end{proof}

Theorem \ref{thm:equivRosu} is an algebraic incarnation of Ro\c{s}u's theorem \cite{Ros3}. Indeed, Ro\c{s}u proves that it is possible to construct the analytic extension of complexified equivariant K-theory only in terms of singular cohomology of fixed loci in $X$. Our Theorem \ref{thm:equivRosu} proves his result without the need to extend scalars by holomorphic functions, and moreover holds over all fields $k$ of characteristic zero. In particular, it answers the question posed by Ro\c{s}u himself of the existence of such an algebraic construction.

\begin{remark}
In particular, if $G$ is a compact Lie group, we could define \emph{adelic} $G$-equivariant K-theory, $\cK_{G}^{ad}(X;k)$, as the Weyl invariants of the adelic $T$-equivariant K-theory, for a maximal torus $T$. Then it would follow immediately that $\cK_{G}^{ad}(X;k)\simeq\cK_G(X)$.
\end{remark}

\section{Geometric presentations for Equivariant Cohomology and K-theory}\label{section:geometry}
In this section we provide a presentation of equivariant K-theory and equivariant singular cohomology of the analytification of a smooth algebraic variety via Hochschild homology counterparts, in terms of objects in derived algebraic geometry. This will be a byproduct of the adelic decomposition for equivariant cohomology and K-theory. 
A similar presentation appears in \cite{MapStI} in the context of rationalized equivariant elliptic cohomology. 
The K-theory case has been analysed by Halpern-Leistner--Pomerleano in \cite{HLPom}; their proof does not make use of adelic descent techniques but rather of Blanc's topological K-theory of a dg-category over $\bC$ \cite{Blanc}. 
They prove the equivalence for smooth quasi-projective schemes acted on by a reductive group so that the quotient admits a \emph{semi-complete KN stratification}, Definition 1.1 of \cite{HLPom}. Our proof is for reductive groups, but we do not require the existence of a semi-complete KN stratification. Moreover, our proof uses more elementary mathematics compared to Halpern-Leistner and Pomerleano's.

The analogous statement for equivariant cohomology --- Theorem \ref{thm:HPlHG} --- is well-known, even though it does not appear in the literature in the formulation we give in this paper. In particular, a version of the theorem for general Artin stacks follows immediately from work of Pantev--To\"en--Vaqui\'e--Vezzosi \cite{PTVV} and Calaque--Pantev--To\"en--Vaqui\'e--Vezzosi \cite{CPTVV}. We include it in the paper only for completeness reasons, and to provide a proof which uses the same techniques that can be  applied to equivariant K-theory and equivariant elliptic cohomology (for this last case see \cite{MapStI}).
\begin{remark}
The construction appearing in \cite{MapStI} for rationalized equivariant elliptic cohomology can be compared, over any field $k$ of characteristic zero, with a \emph{de Rham} variant of the adelic construction of equivariant elliptic cohomology in Section \ref{section:adelicGroj}. In particular, the main result in \cite{MapStI}, Theorem 6.10, would extend over a general field $k$ of characteristic zero to an isomorphism with such a de Rham variant.
\end{remark}

Throughout this section we work over a fixed algebraically closed base field $k$ of characteristic 0.

\subsection{Equivariant K-theory}
In this subsection we prove a small variation of Halpern-Leistner--Pomerleano's Theorem 2.17 in \cite{HLPom}. They prove that, if $X$ is a smooth quasi-projective scheme acted on by an algebraic group $G$ so that $[X/G]$ admits a semi-complete KN stratification (see Definition 1.1 in \cite{HLPom}) there is an equivalence 
$$\mHP(\Perf[X/G])\simeq K^{top}([X/G])\otimes{\bC}$$
between the periodic cyclic homology of $[X/G]$ and Blanc's topological K-theory of the dg-category $\Perf([X/G])$. 

The Proposition we prove in this section is the following:
\begin{proposition}\label{thm:HPK}
Let $X$ be a smooth variety over $\bC$ acted on by an algebraic torus $T$. There is an equivalence of $\bZ_2$-periodic coherent sheaves on $T$
$$\pi_\ast\cHP([X/T])\simeq\cK_{T^{\an}}(X^{\an})$$
which is natural in $X$ with respect to $T$-equivariant maps.
\end{proposition}
In the statement above, $\cHP([X/T])$ denotes the $\bZ_2$-periodic quasi-coherent complex on 
$$T=\Spec \mHP^0([\Spec \bC/T])$$ 
associated to periodic cyclic homology, viewed as a module over $\mHP^0([\Spec \bC/T])$. Similarly, $\cK_{T^\an}(X^\an)$ is the $\bZ_2$-periodic quasi-coherent sheaf on 
$$T=\Spec K_{T^\an}^{0}(\ast)\otimes_{\bZ}\bC$$ 
associated to the $\bC$-rationalized topological K-theory $K_{T^\an}(X^\an)$ viewed as a module over $K_{T^\an}^{0}(\ast)\otimes_{\bZ}\bC$. 

Proposition \ref{thm:HPK} can be extended to reductive groups, by keeping track of the Weyl group action. We do so in Theorem \ref{thm:HPKG}. 
\begin{remark}
The assumption that $X$ is a variety is enough to ensure that $\cHP([X/T])$ belongs to $\Coh(T)^{\bZ_2}\simeq\Perf(T)^{\bZ_2}$. Indeed, this is determined by the finite presentation of the equivariant cohomology of the analytification of $X$ over the base $H_{G^\an}(\ast)$, which is ensured by the fact that, for $X$ a separated scheme of finite type over $\bC$, $X^\an$ is homotopy equivalent to a finite CW-complex (see the discussion in \cite{HChen} in the proof of Corollary 4.3.21).
\end{remark}

\begin{remark}
Let $X$ be a variety acted on by an algebraic torus $T$. Let 
$$q:\cL[X/T]\to T$$
be the structure map. We define:
$$\cHH([X/T]) := q_{\ast}\cO_{\cL[X/T]}$$
The $S^1$-action on the loop space equips this quasi-coherent sheaf with a lift from $\QCoh(T)$ to $\QCoh(T)^{S^1}$ (where $S^1$ acts trivially on $T$). $\cHP([X/T])$ is equivalently the image of $\cHH([X/T])$ with its $S^1$-action in the $\bZ_2$-periodic category of quasi-coherent sheaves 
$$\QCoh(T)^{\bZ_2}$$ 
defined as the Tate construction of $\QCoh(T)$ with respect to a trivial action of $S^1$ on $T$. Such procedure has been defined by Preygel in \cite{PreyTate}. For small $k$-linear categories with an action of $S^1$, it amounts to a base change of the $S^1$-fixed locus from $k[[u]]$ to $k((u))$. For presentable categories there is an additional \emph{regularization} step involving t-structures, as explained in \cite{PreyTate}.
\end{remark}

The arguments we produce to prove Proposition \ref{thm:HPK} are parallel to the ones appearing in the proof of Proposition 6.8 and Theorem 6.9 in \cite{MapStI}. 

\begin{lemma}\label{lemma:HPK1}
Let $X$ be a smooth variety over $\bC$ acted on by an algebraic torus $T$ of rank 1. There is an equivalence of $\bZ_2$-periodic coherent sheaves
$$\pi_\ast \cHP([X/T])\simeq\cK_{T^{\an}}(X^{\an})$$
which is natural in $X$ with respect to $T$-equivariant maps.
\end{lemma} 
\begin{proof}
The adelic decomposition for the sheaf of rationalized equivariant K-theory has been computed in Section \ref{section:adelicsingK}. In the case of periodic cyclic homology, the adelic decomposition follows from Chen's Theorem 4.3.2 in \cite{HChen}: 
$$\pi_\ast\cHP([X/T])_{\widehat{z}}\simeq H_{T^\an}^{\oplus,\ast}(t_0 X^{T(z),\an})\otimes_{\cO(\ft)}\cO_{T,\widehat{z}}$$
and Chen's localization theorem for the loop space, Theorem 3.2.12 in \cite{HChen}:
$$\pi_\ast\cHP([X/T])_{\widehat{\eta}}\simeq H^{\oplus,\ast}(t_0 X^{T,\an})\otimes_{\cO(\ft)}\cO_{T,\eta}$$
The isomorphisms satisfy the compatibility conditions necessary to promote them to isomorphisms of cosimplicial groups. Indeed, the diagram
$$\xymatrix{
\pi_\ast\mHP([X/T])_{\widehat{z}}\ar[r]\ar[d] & H_{dR,T}^{\oplus,\ast}(t_0(X^{T(z)})^\an)\otimes_{\cO(\ft)}\cO_{T,\widehat{z}}\ar[d] \\
\pi_\ast\mHP(t_0(X^T))\otimes_{k}\Frac\cO_{T,\widehat{z}}\ar[r] & H_{dR}^{\oplus,\ast}(t_0(X^T)^\an)\otimes_{k}\Frac\cO_{T,\widehat{z}}
}
$$
commutes as a consequence of the compatibility of Chen and Ben-Zvi--Nadler's HKR theorems 4.3.2 in\cite{HChen} and Proposition 4.4 in \cite{BZNLoopConn} with pullbacks, and the diagram 
$$\xymatrix{
\pi_\ast\mHP([X/T])_{\eta}\ar[r]\ar[d] & H_{dR}^{\oplus,\ast}(t_0(X^{T})^\an)\otimes_{\cO(\ft)}\cO_{T,\eta}\ar[d] \\
\pi_\ast\mHP(t_0(X^T))\otimes_{k}\Frac\cO_{T,\widehat{z}}\ar[r] & H_{dR}^{\oplus,\ast}(t_0(X^T)^\an)\otimes_{k}\Frac\cO_{T,\widehat{z}}
}
$$
commutes as the bottom isomorphism is the tensor product of the top one with $\cO_{T,\widehat{z}}$.

Naturality of the isomorphism is a consequence of the compatibility with $T$-equivariant maps of the HKR isomorphisms provided by Theorem 4.3.2 in \cite{HChen} and Proposition 4.4 in \cite{BZNLoopConn}.
\end{proof}

Proposition \ref{thm:HPK} follows from Lemma \ref{lemma:HPK1} applying induction on the rank of $T$.
\begin{proof}[Proof of Proposition \ref{thm:HPK}]
We first focus on the analysis of chains $\Delta=(x>x_1>\dots>x_p)$ where $x_p$ is a closed point of $T$. In this case, we have an explicit description of the adelic groups in terms of singular cohomology for both $\cHP$ and $\cK$. An application of Huber's Proposition 3.2.1 in \cite{HubAd} to Theorem 4.3.2 from Chen's \cite{HChen} yields
$$\bA_{T}(\Delta,\cHP([X/T]))\simeq H_{T^\an}^{\oplus,\ast}(t_0 (X^{T(x)})^{\an})\otimes_{\cO(\ft)}\cO_{T,\Delta}$$
In the above, the subgroup $T(x)$ is determined according to Definition \ref{def:HHlin}. We also applied the localization theorem (Theorem 3.2.12 in \cite{HChen}) to further restrict from $t_0 X^{T(x_p)}$ to $t_0 X^{T(x)}$. Since $X$ is a variety over $\bC$, the equivariant cohomology of the fixed loci is always finitely generated as a module over the equivariant cohomology of the point.

For K-theory, Lemma \ref{lemma:AdK} gives the same formula:
$$\bA_{T}(\Delta,\cK_{T^\an}(X^{\an}))\simeq H_{T^\an}^{\oplus,\ast}(t_0 (X^{T(x)})^{\an})\otimes_{\cO(\ft)}\cO_{T,\Delta}$$
The cosimplicial structure is induced by pullback maps in equivariant cohomology and the coface/codegeneracy maps for the cosimplicial adelic group of $\cO_{T}$, in both cases. 

We now assume $x_p$ is not a closed point. In this situation, we apply induction on the rank of the torus $T$. Since $T$ is affine, we have decompositions
$$\bA_{T}(\Delta,\cHP([X/T]))\simeq\mHP([t_0 X^{T(x)}/T])\otimes_{\cO(T)}\cO_{T,\Delta}$$
$$\bA_{T}(\Delta,\cK_{T^\an}(X^{\an}))\simeq K_{T^\an}(t_0 (X^{T(x)})^{\an}))\otimes_{\cO(T)}\cO_{T,\Delta}$$
where we are allowed to restrict to the locus fixed by $T(x)$ by the localization theorem respectively for the loop space (Theorem 3.1.12 in \cite{HChen}) and in equivariant K-theory.

Observing that $T(x)$ acts trivially on $t_0 X^{T(x)}$, we can reduce the equivariance group to $T'(x)$:
$$\mHP([t_0 X^{T(x)}/T])\simeq\mHP([t_0 X^{T(x)}/T'(x)])\otimes_{\cO(T'(x))}\cO(T)$$
$$K_{T^\an}(t_0 (X^{T(x)})^{\an}))\simeq K_{T'(x)^\an}(t_0 (X^{T(x)})^{\an}))\otimes_{\cO(T'(x))}\cO(T)$$
Then the inductive hypothesis produces an equivalence of the adelic groups relative to the chain $\Delta$. The cosimplicial structure is controlled by the coface/codegeneracy maps of the adelic group relative to $\cO_{T}$, hence the isomorphism produced above is compatible with the cosimplicial maps induced by removing points $x_i$, $i\in \{1,\dots,p\}$, from the chain $\Delta$. The only coface map that needs to be analysed separately is the one relative to removing the point $x$ from the chain $\Delta$. In this case, naturality with respect to $X$ of the equivalence (which is assumed inductively) allows us to conclude the compatibility.
\end{proof}
\begin{remark}\label{remark:Tatetori}
Chen's Theorem 4.3.2 in \cite{HChen} requires the variety $X$ to be quasi-projective, as this assumption ensures that the Tate construction on the completions recovers the completed Betti cohomology of the analytification --- by forcing the induced $S^1$-action on such completions to be the standard one on Hochschild homology. In Example 4.1.6 Chen shows that for actions of tori, while the induced action on the completions is not the standard one, the Tate construction recovers the expected result. Thus we can omit the quasi-projectivity assumption in the statement of Proposition \ref{thm:HPK}.
\end{remark}

In their paper \cite{HLPom}, Halpern-Leistner and Pomerleano remark that their theorem 2.17 follows from a chain level identification. The same is true in our case. Indeed, we can produce a cochain-level version of rationalized equivariant K-theory via adelic descent, using the same techniques that we use to define rationalized adelic equivariant K-theory in Subsection \ref{ssec:eqvKThy}. Let us call this object $\cCK_T(X)$, for a finite $T$-CW-complex $X$. We impose its ad\`eles at a chain $\Delta=(x,x_1,\dots,x_p)$ with $x_p$ closed to be given by
$$\bA_{T}(\Delta,\cCK_T(X))=C^{\oplus,\ast}_T(X^{T(x)})\otimes_{\cO(\ft)}\bA_{T}(\Delta,\cO_T)$$
where $C^{\oplus,\ast}_T(X^{T(x)})$ are the $\bZ_2$-periodized $T$-equivariant cochains on $X^{T(x)}$ with coefficients in $k$, i.e. the $\bZ_2$-periodization of the cochains on the Borel construction $X^{T(x)}//T$. The cochains $C^{\oplus,\ast}_T(X^{T(x)})$ are a module over $C^{\oplus,0}_T(\ast)$, which is formal, hence $C^{\oplus,0}_T(\ast)\simeq\cO(\ft)$.
The perfect complex $\cCK_{T}(X)$ is then constructed following the same inductive techniques in Subsection \ref{ssec:eqvKThy}.

The same proof of Proposition \ref{thm:HPK} then proves a cochain-level variant:

\begin{proposition}\label{thm:HPKchains}
Let $X$ be a smooth variety over $\bC$ acted on by an algebraic torus $T$. There is an equivalence of objects in $\Perf(T)$
$$\cHP([X/T])\simeq\cCK_{T^{\an}}(X^{\an})$$
which is natural in $X$ with respect to $T$-equivariant maps.
\end{proposition}

\subsubsection{The case of reductive $G$}
Proposition \ref{thm:HPK} can be also stated for quotients by a reductive group $G$, as long as the variety $X$ is quasi-projective. This follows from Theorem 4.3.2 in Chen's \cite{HChen}. The quasi-projectivity hypothesis is required in the case of general reductive $G$ to ensure that the Tate construction on the ad\'eles gives the expected result (see remark \ref{remark:Tatetori}). 

\begin{theorem}\label{thm:HPKG}
Let $X$ be a smooth quasi-projective variety over $\bC$ acted on by a reductive algebraic group $G$. There is an equivalence of $\bZ_2$-periodic coherent sheaves on $G//G=\Spec \cO(G)^G$
$$\pi_\ast\cHP([X/G])\simeq\cK_{G^{\an}}(X^{\an})$$
which is natural in $X$ with respect to $G$-equivariant maps.
\end{theorem}
We recall that $G^\an$-equivariant K-theory is a $\bZ_2$-periodic coherent sheaf on $G//G$, and that there are canonical maps
$$\cL[X/G]\to\cL BG\to\Spec \cO(G)^G=G//G$$

First, observe that there is a comparison map
$$\alpha:\cHP([X/G])\to\cHP([X/T])^W$$
for a maximal torus $T$ and Weyl group $W$, induced by the map of stacks
$$[\cL[X/T]/W]\to\cL[X/G]$$
Here $\cHP([X/T])^W$ denotes the $W$-invariant complex, i.e. the pushforward of $\cHP([X/T])$ along the map
$$[T/W]\to T//W$$
from the stack quotient to the GIT quotient. Moreover, we have that $T//W=G//G$.

\begin{proof}[Proof of Theorem \ref{thm:HPKG}]
In order to prove that $\alpha$ is an equivalence it suffices to show that the completions of $\alpha$ at all closed points of $G//G$ are equivalences. We recall here that such closed points are in bijection with semisimple orbits of the adjoint $G$-action on $G$. Given such a point $z$, we have the following sequence of equivalences:
$$\pi_\ast\cHP([X/G])_{\widehat{z}}\simeq H_{G^z}^{\oplus,\ast}((t_0 X^z))^\an\otimes_{\cO(\fg)^G}\cO_{G//G,\widehat{z}}\simeq\cK_{G^\an}(X^\an)_{\widehat{z}}\simeq$$
$$\simeq (\cK_{T^\an}(X^\an)^W)_{\widehat{z}}\simeq (\pi_\ast\cHP([X/T])^W)_{\widehat{z}}$$
where $G^z$ denotes the centralizer of $z$ in $G$.

In the above, the first equivalence is Theorem 4.3.2 in Chen's \cite{HChen}, the second equivalence is the Atiyah--Segal completion theorem, the third one is the classical fact that, after complexification,
$$K_{G^\an}(X^\an)\otimes_{\bZ}\bC\simeq (K_{T^\an}(X^\an)\otimes_{\bZ}\bC)^W$$
and the last equivalence is a corollary of Proposition \ref{thm:HPK}.

This proves that 
$$\pi_\ast\alpha:\pi_\ast\cHP([X/G])\to\pi_\ast\cHP([X/T])^W$$ 
is an equivalence. Then, as we have that 
$$\pi_\ast\cHP([X/T])^W\simeq\cK_{G^\an}(X^\an)$$
by Proposition \ref{thm:HPK}, we finally obtain
$$\pi_\ast\cHP([X/G])\simeq\cK_{G^\an}(X^\an)$$
\end{proof} 

\begin{remark}\label{remark:HPKGchains}
Similarly to the case of the torus, also Theorem \ref{thm:HPKG} has a chain-level refinement. Define
$$\cC\cK_{G}(X):=\cC\cK_T(X)^W$$
and the identification $\cHP([X/G])\simeq\cC\cK_{G^\an}(X^\an)$ follows immediately from Proposition \ref{thm:HPKchains} using the same arguments. Moreover, we also obtain that
$$\alpha:\cHP([X/G])\to\cHP([X/T])^W$$
is an equivalence.
\end{remark}

Theorem 2.10 in \cite{HLPom} provides and equivalence 
$$K^{top}(\Perf([X/G]))\otimes\bC\simeq K_{G^\an}(X^\an)$$
of presheaves on the category of quasi-projective $G$-schemes valued in the homotopy category of spectra. In the above, $K^{top}$ is Blanc's topological K-theory of a DG-category \cite{Blanc}.
In particular, we can combine this result with Theorem \ref{thm:HPKG} and Remark \ref{remark:HPKGchains} to obtain the following:

\begin{corollary}[lattice conjecture for quotients of varieties]\label{cor:lattice} 
Let $X$ be a smooth quasi-projective variety acted on by a reductive group $G$. Then we have an equivalence
$$K^{top}(\Perf([X/G]))\otimes\bC\simeq\mHP([X/G])$$
\end{corollary}

Corollary \ref{cor:lattice} offers a variant of the lattice conjecture that does not require that the quotient admits a semi-complete KN stratification, as opposed to Theorem 2.17 in \cite{HLPom}. The \emph{lattice conjecture} in this context means the existence of a rational structure on the periodic cyclic homology of a stable $\infty$-category.

\begin{proof}[Proof of Corollary \ref{cor:lattice}]
Theorem 2.10 in \cite{HLPom} gives an equivalence of $\bC$-module spectra
$$K^{top}(\Perf([X/G]))\otimes\bC\simeq K_{G^\an}(X^\an)$$
for smooth quasi-projective schemes $X$ acted on by algebraic groups $G$. On the other hand, Theorem \ref{thm:HPKG} and remark \ref{remark:HPKGchains} give an equivalence
$$\mHP([X/G])\simeq\Gamma(G//G,\cC\cK_{G^{\an}}(X^{\an}))=K_{G^\an}(X^\an)$$
where $\Gamma$ denotes derived global sections. In particular, we conclude
$$K^{top}(\Perf([X/G]))\otimes\bC\simeq\mHP([X/G])$$
\end{proof}

\subsection{Equivariant cohomology}\label{ssection:HPlin}
In this subsection we prove an analogue of Proposition \ref{thm:HPK} for equivariant singular cohomology, using the same technique as in the K-theory case. 
The equivalence at the level of global sections is a consequence of work of Calaque--Pantev--To\"en--Vaqui\'e--Vezzosi \cite{CPTVV}. Their work indeed proves the statement in far grater generality, for any Artin stack in characteristic zero. Our proof is tailored to quotient stacks, and for the statement at the level of sheaves.  Nevertheless, since the base space $\fg//G$ is an affine scheme, we do not get a more general theorem with our techniques. 
We include these results here for completeness. In particular, this allows us to observe that our techniques --- based on adelic descent and induction --- adapt to treat equivariant cohomology as well as equivariant K-theory and equivariant elliptic cohomology, that is discussed in the paper \cite{MapStI}.

First we define the relevant notion of Hochschild homology in this context.
\begin{definition}\label{def:HHlin}
Let $\cX$ be a derived stack. 
We define the \emph{linearized Hochschild homology} of $\cX$ as the global sections
$$\mHHl(\cX):=\cO(T_{\cX}[-1])$$
as an object of $\Mod_k$.
Let $X$ be a scheme acted on by a reductive algebraic group $G$. Let 
$$r:T_{[X/G]}[-1]\to\fg//G$$
be the structure map. We define:
$$\cHHl([X/G]):=r_{\ast}\cO_{T_{[X/G]}[-1]}$$
\end{definition}
\begin{remark}
The \emph{de Rham complex} of $\cX$, $\DR(\cX)$, is a canonical structure of graded mixed complex on the linearized Hochschild homology of $\cX$. This is explained in \cite{PTVV}.
\end{remark}
The next step is defining a periodic cyclic version of linearized Hochschild homology. The first ingredient we need is a global $B\widehat{\bG}_a$-action on the shifted tangent stack
$$T_{[X/G]}[-1]$$
which follows from work of Naef and Safronov \cite{SafNaef}.
This action will be used to perform a Tate construction similar to the procedure used in \cite{MapStI} in the context of equivariant elliptic cohomology, where the relevant action is that of the elliptic curve $E$.
In our setting, the $B\widehat{\bG}_a$-action comes from an identification 
$$T_{\cX}[-1]\simeq\Map{B\widehat{\bG}_a}{\cX}$$
for any derived stack $\cX$ over $k$, which is explained in \cite{SafNaef}. 

\begin{lemma}\label{lemma:trivaction}
Let $G$ be a smooth reductive algebraic group over $k$. The $B\widehat{\bG}_a$-action on 
$$T_{BG}[-1]$$
induces a trivial action of $B\widehat{\bG}_a$ on the affinization $\fg//G$.
\end{lemma}
\begin{proof}
This is a consequence of the fact that any differential p-form on $BG$ over a field of characteristic zero is canonically closed, which implies that the de Rham differential acts trivially. This fact is explained in Section 5 of \cite{ToenDAG}.
The $B\widehat{\bG}_a$-action on $T_{BG}[-1]$ induces a mixed structure on the global sections 
$$\cO(T_{BG}[-1])$$
as there is an identification 
$$\QCoh(BB\widehat{\bG}_a)\simeq\QCoh(BB\bG_a)\simeq\QCoh(BS^1)$$
which can be found in \cite{PreyMF}.
In particular, the triviality of the de Rham differential implies that the mixed structure on $\cO(T_{BG}[-1])$ is trivial, which in turn induces a trivial $B\widehat{\bG}_a$-action on 
$$\Spec \cO(T_{BG}[-1])\simeq\fg//G$$ 
\end{proof}

We follow the same steps as in \cite{MapStI}: Lemma \ref{lemma:trivaction} allows us to promote the action of $B\widehat{\bG}_a$ on the shifted tangent stack 
$$T_{[X/G]}[-1]$$
to an action relative to the base $\Aff{T_{BG}[-1]}\simeq\fg//G$ of the group stack $B\widehat{\bG}_a\times\fg//G$ as an object of the comma $\infty$-category 
$$(\dSt_k)_{/\fg//G}$$
This action induces on the quasi-coherent complex
$$\cHHl([X/G])$$
the structure of a comodule over the Hopf algebra object 
$$\pi_\ast\cO_{B\widehat{\bG}_a\times\fg//G}$$
internal to the $\infty$-category $\QCoh(\fg//G)$, for $\pi:B\widehat{\bG}_a\times\fg//G\to\fg//G$ the canonical projection.
This corresponds to a lift 
$$\cHHl([X/G])\in\QCoh(\fg//G)^{S^1}$$
\begin{definition}
Let $X$ be a scheme acted on by a reductive smooth algebraic group $G$. The \emph{linearized periodic cyclic homology} of $[X/G]$ 
$$\cHPl([X/G])$$
is the image of $\cHHl([X/G])$, as an object of $\QCoh(\fg//G)^{S^1}$, inside the $\bZ_2$-periodic $\infty$-category $\QCoh(\fg//G)^{\bZ_2}$.
\end{definition}
The $\bZ_2$-periodic $\infty$-category $\QCoh(\fg//G)^{\bZ_2}$ is obtained following Preygel's definition of \emph{Tate construction} for $\infty$-categories described in \cite{PreyTate}, equipping $\fg//G$ with a trivial $S^1$-action.

As a preliminary to the theorem, we need a localization formula for the shifted tangent stack of a quotient.  We prove a slightly weaker version of the formula:
\begin{proposition}\label{prop:shiftedloc}
Let $X$ be a smooth variety over $\bC$ acted on by a reductive group $G$. For any closed point $\xi$ of $\fg$ corresponding to a semisimple element, the completion at $\xi$ of the map 
$$T_{[t_0 X^{T(\xi)}/G^{T(\xi)}]}[-1]\to T_{[X/G]}[-1]$$
induced by the inclusion $t_0 X^{T(\xi)}\hookrightarrow X$ of fixed loci is an equivalence.
\end{proposition}
\begin{proof}
This fact is a consequence of the localization formula for the loop space, Theorem 3.1.12 of \cite{HChen}, and the non-equivariant Chern character. Choose $\xi$ a closed point in $\fg$ corresponding to a semisimple element and consider the analytic exponential map
$$\exp:\fg\to G$$
Call $z=\exp(\xi)$.
Localization for the loop space $\cL[X/T]$ implies that we have an equivalence
$$T_{[t_0 X^{T(z)}/G^{T(z)}]}[-1]_{\widehat{\xi}}\simeq T_{[X/G]}[-1]_{\widehat{\xi}}$$
Indeed, as explained in \cite{BZNLoopConn}, the completion at the respective identities of $\fg$ and $G$, i.e. the non-equivariant Chern character 
$$T_{[X/G]}[-1]_{\widehat{0}}\simeq\widehat{T}_{[X/G]}[-1]\simeq\widehat{\cL}[X/G]\simeq\cL[X/G]_{\widehat{1}}$$
is an equivalence.
Moreover, multiplication by points on the respective base --- obtained via the following pullback diagram
$$\xymatrix{
\cL[X/G]_{\widehat{1}}\ar[r]^{\sim}\ar[d] & \cL[X/G]_{\widehat{z}}\ar[d] \\
T_{\widehat{1}}\ar[r]_{\widehat{\mu}_z}^{\sim} & T_{\widehat{z}}
}
$$
is also an equivalence, natural in $[X/G]$, and similarly for the shifted tangent stack. Finally, we apply the localization on the loop space:
$$\cL[t_0 X^{T(z)}/G^{T(z)}]_{\widehat{z}}\simeq T_{[t_0 X^{T(z)}/G^{T(z)}]}[-1]_{\widehat{\xi}}\simeq T_{[X/G]}[-1]_{\widehat{\xi}}\simeq\cL[X/G]_{\widehat{z}}$$ 
The next step is to show that the natural map
\begin{equation}\label{eq:shifteddifference}
T_{[t_0 X^{T(\xi)}/G^{T(\xi)}]}[-1]_{\widehat{\xi}}\to T_{[t_0 X^{T(z)}/G^{T(z)}]}[-1]_{\widehat{\xi}}
\end{equation}
induced by the inclusion $t_0 X^{T(\xi)}\hookrightarrow t_0 X^{T(z)}$ is an equivalence (recall that $T(\xi)\geq T(z)$). Let us analyse the complement $Y=t_0 X^{T(z)}-t_0 X^{T(\xi)}$. Call $S$ the set
$$S=\{K\mbox{ proper subgroup of }T(\xi)\mbox{ containing }T(z)\}$$
We can decompose $Y$ as a union of orbits 
$$Y=\bigcup_{K\in S} \cO_K$$
$Y$ has an action of the centralizer $G^{T(\xi)}$, and every orbit $\cO_K$ is isomorphic to the coset $G^{T(\xi)}/K$. In particular, the shifted tangent stack of $[Y/G^{T(\xi)}]$ decomposes as 
$$T_{[Y/G^{T(\xi)}]}[-1]\simeq\bigcup_{K\in S}T_{\cO_K/K}[-1]\simeq\bigcup_{K\in S} T_{BK}[-1]\simeq\bigcup_{K\in S} [\mathrm{Lie}(K)/K]$$
The conclusion is that the map \eqref{eq:shifteddifference} is an equivalence if $\xi$ does not belong to $\bigcup_{K\in S} \mathrm{Lie}(K)//K$. If $\xi$ belongs to $\mathrm{Lie}(K)$ for some $K\in S$, it follows that $\mathrm{Lie}(T(\xi))\subset \mathrm{Lie}(K)$, as $T(\xi)$ is the smallest subgroup of $T<G$ such that $\xi\in\mathrm{Lie}(T(\xi))$. But if $K\in S$, $K$ is a proper subgroup of $T(\xi)$, hence $\xi$ cannot belong to $\mathrm{Lie}(K)$ for any $K\in S$. 

\end{proof}

We can finally prove the following proposition:
\begin{proposition}\label{thm:HPlH}
Let $X$ be a smooth variety over $\bC$ acted on by an algebraic torus $T$. There is an isomorphism of $\bZ_2$-periodic perfect complexes on $\ft$
$$\cHPl([X/T])\simeq\cC\cH_{T^{\an}}(X^{\an})$$
where on the left-hand side we have the homotopy sheaves of linear periodic cyclic homology, and on the right-hand side the perfect complex on $\ft=\Spec H^{\oplus,0}_{T^\an}(\ast)\simeq\Spec C^{\oplus,0}_{T^\an}(\ast)$ whose global sections are the equivariant singular cochains $C^{\oplus,\ast}_{T^\an}(X^\an)$.
\end{proposition}
\begin{proof}
The proof of this proposition is the same as for Proposition \ref{thm:HPK}, so we give a very short treatment. It is based on the localization formula for the shifted tangent of a quotient, Proposition \ref{prop:shiftedloc}, and on the same inductive argument of the proof of Proposition \ref{thm:HPK}. In this situation, the localization theorem for the shifted tangent stack of $[X/T]$ implies
$$\bA_{\ft}(\Delta, \cHPl([X/T]))=\mHP([t_0 X^{T(\xi)}/T])\otimes_{\cO(\ft)}\cO_{\ft,\widehat{\Delta}}$$
for a chain $\Delta=(\xi>\xi_1>\dots>\xi_0)$ on $\ft$, where $\xi_0$ is a closed point. In this context, the group $T(\xi)$ is the one appearing in Definition \ref{def:dictH}.
On the other hand, for the complex $\cH_{T^{\an}}(X^{\an})$ we have
$$\bA_{\ft}(\Delta, \cH_{T^{\an}}(X^{\an}))=C^{\oplus,\ast}_{T^\an}((t_0 X^{T(\xi)})^\an)\otimes_{\cO(\ft)}\cO_{\ft,\widehat{\Delta}}$$

If $\xi_0$ is not a closed point, and to deal with the cosimplicial structure, we use the exact same inductive argument as in the proof of Proposition \ref{thm:HPK}. Indeed, if $\mathrm{rk}(T)=1$ the conclusion is immediate.
\end{proof}

\begin{remark}
As explained in Remark \ref{remark:Tatetori}, we can omit the quasi-projectivity assumption on $X$ in the statement of Proposition \ref{thm:HPlH}.
\end{remark}

\subsubsection{The case of reductive $G$}
Similarly to Proposition \ref{thm:HPK}, we can state Proposition \ref{thm:HPlH} for quotients by a reductive group $G$ if the variety $X$ is quasi-projective. To do so, we only need to replace the loop space with the shifted tangent bundle in the statements and arguments that led to the proof of Proposition \ref{thm:HPKG}. The result is

\begin{theorem}\label{thm:HPlHG}
Let $X$ be a smooth quasi-projective variety over $\bC$ acted on by a reductive algebraic group $G$. There is an equivalence of $\bZ_2$-periodic coherent sheaves on $\fg//G=\Spec \cO(\fg)^G$
$$\pi_\ast\cHPl([X/G])\simeq\cH_{G^{\an}}(X^{\an})$$
which is natural in $X$ with respect to $G$-equivariant maps.
\end{theorem}
\begin{proof}
The proof is the same as in the case of Theorem \ref{thm:HPKG}. We have a comparison map
$$\cHPl([X/G])\to\cHPl([X/T])^W$$ 
induced by a map at the level of stacks 
$$[T_{[X/T]}[-1]/W]\to T_{[X/G]}[-1]$$
The completions at closed points of $\fg//G$ are computed by the localization formula (Proposition \ref{prop:shiftedloc}):
$$\pi_\ast\cHPl([X/G])_{\widehat{\xi}}\simeq H_{G^{T(\xi)}}^{\oplus,\ast}((t_0 X^{T(\xi)}))^\an\otimes_{\cO(\fg)^G}\cO_{\fg//G,\widehat{\xi}}\simeq\cH_{G^\an}(X^\an)_{\widehat{\xi}}\simeq$$
$$\simeq (\cH_{T^\an}(X^\an)^W)_{\widehat{\xi}}\simeq (\pi_\ast\cHPl([X/T])^W)_{\widehat{\xi}}$$
and this allows us to conclude that 
$$\pi_\ast\cHPl([X/G])\xrightarrow{\sim}\pi_\ast\cHPl([X/T])^W$$
from which the theorem follows from an application of Proposition \ref{thm:HPlH}.
\end{proof}
\begin{remark}
Similarly to the case of $\cHP$ it also holds that 
$$\cHPl([X/G])\xrightarrow{\sim}\cHPl([X/T])^W$$ 
\end{remark}

\end{document}